\documentclass[11pt]{amsart}

\usepackage{amsmath}
\usepackage{amssymb}
\usepackage{amscd}
\usepackage{amsthm}
\usepackage{hyperref}
\usepackage{colortbl}

\usepackage{mathtools}
\usepackage{extarrows}

\usepackage[all]{xy}

\usepackage{tikz-cd}
\newtheorem{thm}{Theorem}[section]

\newtheorem{cor}[thm]{Corollary}
\newtheorem{conj}[thm]{Conjecture}

\theoremstyle{remark}

\theoremstyle{definition}

\numberwithin{equation}{section}

\allowdisplaybreaks[4]

\begin{document}

\vfuzz0.5pc
\hfuzz0.5pc % Don't bother to report 
          % overfull boxes if overage is < 6pt

\newcommand{\claimref}[1]{Claim \ref{#1}}
\newcommand{\thmref}[1]{Theorem \ref{#1}}
\newcommand{\propref}[1]{Proposition \ref{#1}}
\newcommand{\lemref}[1]{Lemma \ref{#1}}
\newcommand{\coref}[1]{Corollary \ref{#1}}
\newcommand{\remref}[1]{Remark \ref{#1}}
\newcommand{\conjref}[1]{Conjecture \ref{#1}}
\newcommand{\questionref}[1]{Question \ref{#1}}
\newcommand{\defnref}[1]{Definition \ref{#1}}
\newcommand{\secref}[1]{\S \ref{#1}}
\newcommand{\ssecref}[1]{\ref{#1}}
\newcommand{\sssecref}[1]{\ref{#1}}

\newcommand{\RED}{{\mathrm{red}}}
\newcommand{\tors}{{\mathrm{tors}}}
\newcommand{\eq}{\Leftrightarrow}

\newcommand{\mapright}[1]{\smash{\mathop{\longrightarrow}\limits^{#1}}}
\newcommand{\mapleft}[1]{\smash{\mathop{\longleftarrow}\limits^{#1}}}
\newcommand{\mapdown}[1]{\Big\downarrow\rlap{$\vcenter{\hbox{$\scriptstyle#1$}}$}}
\newcommand{\smapdown}[1]{\downarrow\rlap{$\vcenter{\hbox{$\scriptstyle#1$}}$}}

\newcommand{\A}{{\mathbb A}}
\newcommand{\I}{{\mathcal I}}
\newcommand{\J}{{\mathcal J}}
\newcommand{\CO}{{\mathcal O}}

\newcommand{\C}{{\mathcal C}}
\newcommand{\BC}{{\mathbb C}}
\newcommand{\BQ}{{\mathbb Q}}
\newcommand{\m}{{\mathcal M}}
\newcommand{\h}{{\mathcal H}}
\newcommand{\Z}{{\mathbb Z}}
\newcommand{\BZ}{{\mathbb Z}}
\newcommand{\W}{{\mathcal W}}
\newcommand{\Y}{{\mathcal Y}}

\newcommand{\T}{{\mathcal T}}
\newcommand{\BP}{{\mathbb P}}
\newcommand{\CP}{{\mathcal P}}
\newcommand{\G}{{\mathbb G}}
\newcommand{\BR}{{\mathbb R}}
\newcommand{\D}{{\mathcal D}}
\newcommand{\DD}{{\mathcal D}}
\newcommand{\LL}{{\mathcal L}}
\newcommand{\f}{{\mathcal F}}
\newcommand{\E}{{\mathcal E}}
\newcommand{\BN}{{\mathbb N}}
\newcommand{\N}{{\mathcal N}}
\newcommand{\K}{{\mathcal K}}
\newcommand{\R} {{\mathbb R}}
\newcommand{\PP}{{\mathbb P}}
\newcommand{\Pp}{{\mathbb P}}
\newcommand{\BF}{{\mathbb F}}
\newcommand{\QQ}{{\mathcal Q}}
\newcommand{\closure}[1]{\overline{#1}}
\newcommand{\EQ}{\Leftrightarrow}
\newcommand{\imply}{\Rightarrow}
\newcommand{\isom}{\cong}
\newcommand{\embed}{\hookrightarrow}
\newcommand{\tensor}{\mathop{\otimes}}
\newcommand{\wt}[1]{{\widetilde{#1}}}
\newcommand{\ol}{\overline}
\newcommand{\ul}{\underline}

\newcommand{\bs}{{\backslash}}
\newcommand{\CS}{{\mathcal S}}
\newcommand{\cA}{{\mathcal A}}
\newcommand{\Q}{{\mathbb Q}}
\newcommand{\F}{{\mathcal F}}
\newcommand{\sing}{{\text{sing}}}
\newcommand{\U} {{\mathcal U}}
\newcommand{\B}{{\mathcal B}}
\newcommand{\X}{{\mathcal X}}

% Javier-Macro
\newcommand{\ECS}[1]{E_{#1}(X)}
\newcommand{\CV}[2]{{\mathcal C}_{#1,#2}(X)}

\newcommand{\rank}{\mathop{\mathrm{rank}}\nolimits}
\newcommand{\codim}{\mathop{\mathrm{codim}}\nolimits}
\newcommand{\Ord}{\mathop{\mathrm{Ord}}\nolimits}
\newcommand{\Var}{\mathop{\mathrm{Var}}\nolimits}
\newcommand{\Ext}{\mathop{\mathrm{Ext}}\nolimits}
\newcommand{\EXT}{\mathop{{\mathcal E}\mathrm{xt}}\nolimits}
\newcommand{\Pic}{\mathop{\mathrm{Pic}}\nolimits}
\newcommand{\Spec}{\mathop{\mathrm{Spec}}\nolimits}
\newcommand{\Jac}{\mathop{\mathrm{Jac}}\nolimits}
\newcommand{\Div}{\mathop{\mathrm{Div}}\nolimits}
\newcommand{\sgn}{\mathop{\mathrm{sgn}}\nolimits}
\newcommand{\supp}{\mathop{\mathrm{supp}}\nolimits}
\newcommand{\Hom}{\mathop{\mathrm{Hom}}\nolimits}
\newcommand{\Sym}{\mathop{\mathrm{Sym}}\nolimits}
\newcommand{\nilrad}{\mathop{\mathrm{nilrad}}\nolimits}
\newcommand{\Ann}{\mathop{\mathrm{Ann}}\nolimits}
\newcommand{\Proj}{\mathop{\mathrm{Proj}}\nolimits}
\newcommand{\mult}{\mathop{\mathrm{mult}}\nolimits}
\newcommand{\Bs}{\mathop{\mathrm{Bs}}\nolimits}
\newcommand{\Span}{\mathop{\mathrm{Span}}\nolimits}
\newcommand{\IM}{\mathop{\mathrm{Im}}\nolimits}
\newcommand{\Hol}{\mathop{\mathrm{Hol}}\nolimits}
\newcommand{\End}{\mathop{\mathrm{End}}\nolimits}
\newcommand{\CH}{\mathop{\mathrm{CH}}\nolimits}
\newcommand{\Exec}{\mathop{\mathrm{Exec}}\nolimits}
\newcommand{\SPAN}{\mathop{\mathrm{span}}\nolimits}
\newcommand{\birat}{\mathop{\mathrm{birat}}\nolimits}
\newcommand{\cl}{\mathop{\mathrm{cl}}\nolimits}
\newcommand{\rat}{\mathop{\mathrm{rat}}\nolimits}
\newcommand{\Bir}{\mathop{\mathrm{Bir}}\nolimits}
\newcommand{\Rat}{\mathop{\mathrm{Rat}}\nolimits}
\newcommand{\aut}{\mathop{\mathrm{aut}}\nolimits}
\newcommand{\Aut}{\mathop{\mathrm{Aut}}\nolimits}
\newcommand{\eff}{\mathop{\mathrm{eff}}\nolimits}
\newcommand{\nef}{\mathop{\mathrm{nef}}\nolimits}
\newcommand{\amp}{\mathop{\mathrm{amp}}\nolimits}
\newcommand{\DIV}{\mathop{\mathrm{Div}}\nolimits}
\newcommand{\Bl}{\mathop{\mathrm{Bl}}\nolimits}
\newcommand{\Cox}{\mathop{\mathrm{Cox}}\nolimits}
\newcommand{\NE}{\mathop{\mathrm{NE}}\nolimits}
\newcommand{\NM}{\mathop{\mathrm{NM}}\nolimits}
\newcommand{\Gal}{\mathop{\mathrm{Gal}}\nolimits}
\newcommand{\coker}{\mathop{\mathrm{coker}}\nolimits}
\newcommand{\ch}{\mathop{\mathrm{ch}}\nolimits}
\newcommand{\DM}{\mathop{\mathrm{DM}}\nolimits}
\newcommand{\DMgm}{\mathop{\mathrm{DM}_{\mathrm{gm}}}\nolimits}
\newcommand{\DMgmeff}{\mathop{\mathrm{DM}_{\mathrm{gm}}^{\mathrm{eff}}}\nolimits}
\newcommand{\DMgmab}{\mathop{\mathrm{DM}_{\mathrm{gm}}^{\mathrm{ab}}}\nolimits}
\newcommand{\MHS}{\mathop{\mathrm{MHS}}\nolimits}
\newcommand{\Alb}{\mathop{\mathrm{Alb}}\nolimits}
\newcommand{\NS}{\mathop{\mathrm{NS}}\nolimits}
\newcommand{\InHom}{\mathop{\mathcal{H}om}\nolimits}
\newcommand{\BM}{\mathop{\mathrm{BM}}\nolimits}
\newcommand{\pt}{\mathop{\mathrm{pt}}\nolimits}
\newcommand{\Gr}{\mathop{\mathrm{Gr}}\nolimits}
\newcommand{\im}{\mathop{\mathrm{Im}}\nolimits}

\def\O{\mathcal{O}}
\def\M{\mathcal{M}}
\def\C{\mathbb{C}}
\def\CO{\mathcal{O}}
\def\P{\mathbb{P}}

\def\mysep{.5\baselineskip plus .2\baselineskip minus .2\baselineskip}

\title{$\A^1$-equivalence of zero cycles on surfaces II}

\author{Qizheng Yin}
\address[Yin]{Departement Mathematik\\ETH Z\"urich\\R\"amistrasse 101\\8092 Z\"urich\\Switzerland}
\email{qizheng.yin@math.ethz.ch}

\author{Yi Zhu}
\address[Zhu]{Pure Mathematics\\Univeristy of Waterloo\\Waterloo, ON N2L3G1\\ Canada}
\email{yi.zhu@uwaterloo.ca}

%%%\date{Nov 20, 2001}

\date{\today}

% % 
 \thanks{Q.~Y.~was supported by the grant ERC-2012-AdG-320368-MCSK}
% % and Engineering Research Council of Canada.}
% % 
% % 

% % 
 \begin{abstract}
Using recent developments in the theory of mixed motives, we prove that the log Bloch conjecture holds for an open smooth complex surface if the Bloch conjecture holds for its compactification. This verifies the log Bloch conjecture for all $\Q$-homology planes and for open smooth surfaces which are not of log general type.
 \end{abstract}
% % 

\maketitle

\section{Introduction}

\noindent Throughout this paper, we work with varieties over the complex numbers.

\subsection{Statement of the main theorem}
Let $U$ be a smooth quasiprojective algebraic variety. Let $$a:h_0(U)^0\to \Alb(U)$$ be the Albanese morphism from the zeroth Suslin homology of degree zero to the Albanese variety of $U$, and let $T(U):=\ker(a)$ be the Albanese kernel. When $U$ is projective, $h_0(U)$ reduces to the Chow group of zero cycles $\CH_0(U)$. Indeed, we get the classical Albanese map.

In dimension one, the Albanese morphism is well-understood by the classical work of Abel-Jacobi in the projective case, and by Rosenlicht in the open case. 

\begin{thm}[Abel-Jacobi, Rosenlicht \cite{Rosenlicht52,Rosenlicht54}]\label{thm:AJ}
When $\dim U=1$, the Albanese morphism is an isomorphism. 
\end{thm}

The higher-dimensional analogue of Theorem \ref{thm:AJ} is much more subtle, although the torsion part of the Albanese morphism is known.

\begin{thm}[Ro{\u\i}tman \cite{Roitman80}, Spie{\ss}-Szamuely \cite{Szamuely03}] \label{thm:torsion}
In arbitrary dimension, the Albanese morphism induces an isomorphism on torsion subgroups. 
\end{thm}

In this paper, we study the two-dimensional case. In one direction, the log Mumford theorem says that the Albanese morphism fails to be injective as long as $p_g(U)\neq 0$.  

\begin{thm}[Mumford \cite{Mumford68}, Zhu \cite{logmumford}]
Let $U$ be a smooth algebraic surface with $p_g(U)\neq0$. Then $T(U)$ is infinite-dimensional.
\end{thm}

In the other direction, we expect the following conjecture. When $U$ is projective, it is famously known as the Bloch conjecture \cite{Bloch80}. 

\begin{conj}[Log Bloch conjecture]
Let $U$ be a smooth algebraic surface with $p_g(U)=0$. Then $$T(U)=0.$$
\end{conj}

Using recent developments in the theory of mixed motives \cite{Ayoub09,Ayoub11,BVK,Ayoub15}, we prove the following theorem. 

\begin{thm}\label{thm:1}
Let $(X,D)$ be a log smooth projective surface pair with interior $U$. If $p_g(U)=0$, in particular, $p_g(X)=0$ as well, then the log Bloch conjecture holds for $U$ if and only if it holds for $X$. 
\end{thm}

Since the Bloch conjecture holds for any smooth projective surface $X$ with $\kappa(X)\le 1$ \cite{Bloch76}, our main theorem yields the following corollary.

\begin{cor}\label{cor}
The log Bloch conjecture holds for $U$ if $\kappa(X)\le 1$. \qed
\end{cor}

Since $\kappa(X)\le\kappa(U)$, Corollary \ref{cor} generalizes the result of Bloch-Kas-Lieberman \cite{Bloch76} to open surfaces of $\kappa(U)\le 1$. It also covers the second author's previous result \cite{logmumford} on the log Bloch conjecture for $\kappa(U)=-\infty$.

Further, we may apply Theorem \ref{thm:1} to the case where $X$ is of general type and the Bloch conjecture is true. The Bloch conjecture holds in a great number of cases; see \cite{BCP, Pedrini13,Voisin14} for recent developments.

\subsection{Applications of Theorem \ref{thm:1} and Corollary \ref{cor}}
%To verify the Bloch conjecture, Bloch-Kas-Lieberman \cite{Bloch76} uses the Enriques classification of smooth projective surfaces with $\kappa \le 1$ . 

The birational geometry of open surfaces is developed by Kawamata \cite{Kawamata1979}, while it is almost impossible to hope for a complete classification even for $\kappa(U) \le 1$. We would like to focus on three special classes of surfaces whose geometry is extremely complicated.

%Given a smooth open surface $U$ with a log smooth compactification $(X,D)$, we call it a \emph{nonsingular triple} $(X,D;U)$. 

%Since we have $\kappa(X)\le\kappa(U)$, consider the table below dividing nonsingular triples into three different types:
%\begin{center}
%\begin{tabular}{|c|c|c|}
%	\hline
%	Type &$\kappa(U)$ & $\kappa(X)$ \\
%	\hline\hline
%	I& $\le 1$ &  $\le 1$ \\\hline
%	II&$2$&  $\le 1$\\
%	\hline 
%	III& $2$ & $2$ \\
%	\hline
%\end{tabular}

%\end{center}

%\subsubsection{\textbf{Example 1}: (Special $U$ when $\kappa(U)=-\infty$)}
%\begin{cor}
%	The log Bloch conjecture holds for open surfaces $U$ with $\kappa(U)\le 1$.\qed
%\end{cor}

%\begin{example}[$\kappa(U)=-\infty$]

\vspace{\mysep}
\noindent \textbf{Example 1} ($\kappa(U)=-\infty$)\textbf{: log del Pezzo surfaces}

\vspace{\mysep} 
Let $U$ be the smooth locus of a singular del Pezzo surface of Picard number one with at worst quotient singularities. In general, such singular del Pezzo's form an unbounded family. Partial classifications are obtained in \cite{KM} with more than sixty exceptional collections. A difficult theorem of Keel-McKernan \cite{KM} states that $U$ is log rationally connected. In particular, it implies the log Bloch conjecture for $U$ \cite[Prop.~4.3]{logmumford}.

Since Theorem \ref{thm:1} and Corollary \ref{cor} do not depend on Keel-McKernan's result, we give a new proof of the following result.

\begin{cor}\label{cor:d}
With the notation as above, we have $h_0(U)=\Z$.
\end{cor}

%\end{example}

\noindent \textbf{Example 2} ($\kappa(U)=0$)\textbf{: log Enriques surfaces}

\vspace{\mysep} 	 
A projective normal surface $Y$ is said to be a \emph{log Enriques surface} if 
\begin{enumerate}
	\item $Y$ has at worst quotient singularities;
	\item $NK_Y\sim \O_Y$ for some positive integer $N$;
	\item $\dim H^1(Y,\O_Y)=0$. 
\end{enumerate}

Since $K_Y$ is $\Q$-Cartier, we define the \emph{index} $I$ of $Y$ to be the smallest positive integer that $IK_Y \sim \O_Y$. By the work of Kawamata \cite{Kawamata1979}, Tsunoda \cite{Tsunoda83}, and Zhang \cite{logE}, the index is bounded by $66$, while classically (when $Y$ is smooth projective) it is bounded by $6$. 

\begin{cor}\label{cor:e}
Let $U$ be the smooth locus of a log Enriques surface of index $\ge 2$ defined as above. Then $h_0(U)=\Z$.
\end{cor}

Log Enriques surfaces are partially classified in \cite{logE,logE2,Kud02,Kud04}. There are more than $1000$ examples of log Enriques surfaces with $\delta$-invariant $2$ \cite{Kud02}.

\proof[Proof of Corollaries \ref{cor:d}, \ref{cor:e}] Let $(X,D)$ be a minimal log resolution of $U$. By Corollary \ref{cor}, the log Bloch conjecture holds in both cases. It suffices to show $q(U)=0$.  Since $D$ is the exceptional set of the resolution of quotient singularities, we have $q(U)=q(X)$. Now the del Pezzo case follows from \cite[Lem.~1.1 (3)]{Zhang89} and the Enriques case from \cite[Lem.~1.2]{logE}. \qed

\vspace{\mysep}
\noindent \textbf{Example 3: $\Q$-homology planes}

\vspace{\mysep}
A smooth surface $U$ is a \emph{$\Q$-homology plane} if $H^i(U,\Q)=H^i(\A^2,\Q)$ for any $i$. A $\Q$-homology plane can have log Kodaira dimension $-\infty$, $0$, $1$, or~$2$. Ramanujam \cite{Ramanujam71} constructed the first homology plane of log general type which is topologically contractible. They are classified for log Kodaira dimension $\le 1$, but there is no thorough classification of $\Q$-homology planes of log general type \cite[Sect.~3.4]{Miyanishi}. 

Since all $\Q$-homology planes are rational \cite{Gurjar99}, Corollary \ref{cor} implies:

\begin{cor}
Let $U$ be a $\Q$-homology plane. Then the log Bloch conjecture holds, that is, $h_0(U)=\Z$.\qed
\end{cor}

The Bloch conjecture for fake projective planes remains unknown.

%\begin{cor}	Assume that the log Bloch conjecture holds for a smooth projective surface $X$. Then it holds for 
%	\begin{enumerate}
%		\item $X$ deleting finitely many points;
%		\item the smooth locus of the canonical model of $X$ when $\kappa(X)=2$;
%		\item the smooth locus of the log minimal model of $(X,D)$ with $p_g(U)=0$;
%		\item the smooth locus of the log canonical model of $(X,D)$ when $\kappa(U)=2$ and $p_g(U)=0$ (not sure, need to check reference.)
%	\end{enumerate}\qed
%\end{cor}

\subsection{Ideas from mixed motives}
The proof of our main theorem has two main ingredients. One is the work of Ayoub, Barbieri-Viale, and Kahn \cite{Ayoub09,Ayoub11,BVK} on the derived category of $1$-motives, especially the construction of a derived Albanese functor. The use is twofold: first, it gives a motivic interpretation of the Albanese morphism, allowing us to apply tools from the theory of mixed motives. Second, it provides a way to eliminate ``easy'' pieces of the motive of $U$ (essentially $1$-motives) while keeping track of the homological realization.

The other ingredient is the famous conservativity conjecture; see \cite{Ayoub15}. Regarded as one of the key conjectures in the study of motives, it notably says that a geometric motives is trivial if and only if its homological realization is trivial. By truncating the motive of $U$ using the derived Albanese functor, we arrive at a motive which has trivial homological realization and whose motivic homology controls the Albanese kernel $T(U)$. Therefore, the conservativity conjecture implies the log Bloch conjecture for $U$. Part of our main theorem then follows from a special case of the conservativity conjecture proven by Wildeshaus \cite{Wilde15}.

Further, it is worth mentioning that the work of Bondarko-Sosnilo \cite{BS14}, if well-interpreted, might also lead to our results.

\subsection{Notation}
A \emph{log pair} $(X,D)$ means a variety $X$ with a reduced Weil divisor $D$. We say that $(X,D)$ is \emph{log smooth} if $X$ is smooth and $D$ is a simple normal crossing divisor on $X$. A log pair is projective if the ambient variety is projective. 
	
Given any smooth quasiprojective variety $U$, by the resolution of singularities, we may choose a log smooth projective compactification $(X,D)$ with interior $U$. We use $\kappa(X,D)$ to denote the \emph{log Kodaira dimension}. We define the \emph{log geometric genus} $p_g(X,D):=\dim H^0\big(\Omega^{\dim X}_X(\log D)\big)$ and the \emph{log irregularity} $q(X,D):=\dim H^0\big(\Omega^{1}_X(\log D)\big)$. Since they do not depend on the compactification, we may write $\kappa(U)$, $p_g(U)$, and $q(U)$ as well.

\subsection*{Acknowledgment}
We would like to thank Joseph Ayoub for explaining his results. We thank Qile Chen and Javier Fres\'an for helpful discussions. This work was initiated during the AMS Summer Institute in Algebraic Geometry at the University of Utah, 2015. The authors would like to thank the Summer Institute for its hospitality and inspiring environment. 

\section{Preliminaries}

\noindent By Theorem \ref{thm:torsion}, it suffices to consider the Albanese morphism with $\Q$-coefficients. From now on, all (co)homology, cycle groups, and motives are taken with $\Q$-coefficients.

\subsection{Mixed motives and conservativity}
We refer to \cite{VSF} for Voevodsky's theory of mixed motives. Since we work with $\Q$-coefficients, the categories of mixed motives in the Nisnevich and \'etale topologies are equivalent, with or without transfers; see \cite{Ayoub14}.

Let $\DMgm$ denote the triangulated category of geometric motives, and let $\DMgmeff$ denote the triangulated category of effective geometric motives. We follow the homological convention. The unit object of $\DMgm$ is denoted by $\Q(0)$, or simply $\Q$, and the Tate object $\Q(1)$. Given an object $M \in \DMgm$, its dual object $\InHom_{\DMgm}(M, \Q)$ is denoted by $M^\vee$. The motive of a smooth variety $Y$ is denoted by $M(Y) \in \DMgmeff$.

The $i$-th motivic homology of $M \in \DMgm$ is defined to be
$$h_i(M) = \Hom_{\DMgm}\big(\Q[i], M\big).$$
For $M = M(Y)$, this recovers the $i$-th Suslin homology $h_i(Y) = h_i\big(M(Y)\big)$.

Further, we refer to \cite{Huber} for the Hodge realization functor
$$R^H: \DMgm \to D^b(\MHS).$$
Composing with the forgetful functor $D^b(\MHS) \to D^b(\Q)$, we obtain the Betti realization $R^B: \DMgm \to D^b(\Q)$. Recall the statement of the conservativity conjecture.

\begin{conj}[see {\cite[Conj.~2.1]{Ayoub15}}]
The Betti realization functor $R^B$ is conservative. In other words, a morphism $f: M \to N$ in $\DMgm$ is an isomorphism if and only if $R^B(f): R^B(M) \to R^B(N)$ is an isomorphism. 
\end{conj}

Using consequences of the standard conjecture D for abelian varieties \cite{AK}, Kimura-O'Sullivan finiteness \cite{Kimura05}, and Bondarko's weight structures \cite{Bond09,Bond10}, Wildeshaus proved the following special case of the conservativity conjecture.

\begin{thm}[Wildeshaus {\cite[Th.~1.12]{Wilde15}}] \label{thm:cons}
Let $\DMgmab \subset \DMgm$ denote the smallest triangulated subcategory containing the motives of smooth curves and closed under direct summands, tensor products, and duality. Then the restriction of $R^B$ to $\DMgmab$ is conservative.
\end{thm}

By introducing $\DMgmab$, we may reformulate our main theorem as follows.

\begin{thm} \label{thm:2}
Under the assumption as in Theorem \ref{thm:1}, the following three conditions are equivalent:
\begin{enumerate}
\item $T(U) = 0$;

\item $T(X) = 0$;

\item $M(U)$, $M(X) \in \DMgmab$.
\end{enumerate}
\end{thm}

\subsection{Derived category of \texorpdfstring{$1$}{1}-motives}

We shall mainly follow the book of Barbieri-Viale-Kahn \cite{BVK}. Let $\M_1$ denote Deligne's category of $1$-motives \cite{Hodgeiii} with $\Q$-coefficients. By \cite[Th.~3.4.1]{Orgogo}, the bounded derived category $D^b(\M_1)$ can be naturally identified with the thick triangulated subcategory of $\DMgmeff$ generated by the motives of smooth curves, denoted by $d_{\leq 1}\DMgmeff$. The identification is compatible with realizations \cite{Vologod}. For simplicity we always make this identification.

One of the main results of \cite{BVK} is the construction of a \emph{derived Albanese functor}.

\begin{thm}[{\cite[Cor.~6.2.2]{BVK}}]
The inclusion $d_{\leq 1}\DMgmeff \hookrightarrow \DMgmeff$ admits a left adjoint
$$L\Alb: \DMgmeff \to d_{\leq 1}\DMgmeff.$$
\end{thm}

We list a number of results and facts about the functor $L\Alb$, which will be used in the proof of our main theorem. To begin with, when $Y$ is a smooth variety, we write $L\Alb(Y) = L\Alb\big(M(Y)\big)$. Then the natural morphism $M(Y) \to L\Alb(Y)$ induces a morphism in motivic homology
\begin{equation} \label{eq:motalb}
h_0(Y) \to h_0\big(L\Alb(Y)\big).
\end{equation}
By \cite[Lem.~13.4.2]{BVK}, we have
$$h_0\big(L\Alb(Y)\big)^0 = \Alb(Y) \otimes \Q,$$
and the degree zero part of \eqref{eq:motalb} coincides with the Albanese morphism.

The next statement concerns the Hodge realization of $L\Alb(M)$ for $M \in \DMgmeff$. Recall that a mixed Hodge structure $H$ is \emph{effective} if the $(i, j)$-th part of the weight-graded piece $\Gr_{i + j}^W H$ vanishes unless $i$, $j \leq 0$. Given an effective mixed Hodge structure $H$, let $H_{\leq 1}$ denote the maximal quotient of~$H$ of weights $\geq -2$ and of types $(0, 0)$, $(0, -1)$, $(-1, 0)$, and $(-1, -1)$.

\begin{thm}[{\cite[Th.~15.3.1]{BVK}}] \label{thm:real}
For $M \in \DMgmeff$, the morphism $M \to L\Alb(M)$ induces isomorphisms
$$H_i\big(R^H(M)\big)_{\leq 1} \xrightarrow{\sim} H_i\Big(R^H\big(L\Alb(M)\big)\Big).$$
\end{thm}

The theorem above applies to $L\Alb(Y)$ and also to the Borel-Moore variant of $L\Alb(Y)$. Let $M^c(Y) \in \DMgmeff$ denote the motive of $Y$ with compact support. By \cite[Ch.~5, Th.~4.3.7]{VSF}, there is an isomorphism
$$M^c(Y) \simeq M(Y)^\vee(\dim Y)[2\dim Y].$$
We write $L\Alb^c(Y) = L\Alb\big(M^c(Y)\big)$.

\begin{cor}[{\cite[Cor.~15.3.2]{BVK}}] \label{cor:real}
By Theorem \ref{thm:real}, we have
$$H_i\Big(R^H\big(L\Alb(Y)\big)\Big) =
\begin{cases}
H_0(Y, \Q) & i = 0 \\
H_1(Y, \Q) & i = 1 \\
H_2(Y, \Q)_{\leq 1} & i = 2 \\
0 & i < 0 \text{ or } i > 2
\end{cases}$$
and
$$H_i\Big(R^H\big(L\Alb^c(Y)\big)\Big) =
\begin{cases}
H_0^{\BM}(Y, \Q) & i = 0 \\
H_1^{\BM}(Y, \Q) & i = 1 \\
H_i^{\BM}(Y, \Q)_{\leq 1} & 2 \leq i \leq \dim Y + 1 \\
0 & i < 0 \text{ or } i > \dim Y + 1.
\end{cases}$$
\end{cor}

Finally, we recall the fact that $\M_1$ is of cohomological dimension one \cite[Prop.~3.2.4]{Orgogo}. Hence, all elements in $D^b(\M_1)$ can be represented by complexes with zero differentials. In particular, we have
$$L\Alb(Y) \simeq \bigoplus_{i = 0}^2L_i\Alb(Y)[i] \ \text{ and } \ L\Alb^c(Y) \simeq \bigoplus_{i = 0}^{\dim Y + 1} L_i\Alb^c(Y)[i],$$
with $L_i\Alb(Y)$, $L_i\Alb^c(Y) \in \M_1$; see \cite[Cor.~9.2.3, Prop.~10.6.2]{BVK}. When $\dim Y = 1$, this gives the ``Chow-K\"unneth'' decomposition of $M(Y)$ \cite[Cor~11.1.1]{BVK}
\begin{equation} \label{eq:ckcurve}
M(Y) \simeq L\Alb(Y) \simeq \bigoplus_{i = 0}^2L_i\Alb(Y)[i].
\end{equation}

\section{Proof of the main theorem}

\noindent In this section we prove our main theorem, that is, Theorem \ref{thm:2}.

\subsection{Proof of (1) \texorpdfstring{$\Rightarrow$}{=>} (2) \texorpdfstring{$\Rightarrow$}{=>} (3)}
For (1) $\Rightarrow$ (2), consider a partial compactification $U \subset Y \subset X$ such that $C = Y \setminus U$ is a smooth curve. By induction, it suffices to show that $T(U) = 0$ implies $T(Y) = 0$.

Recall the Gysin distinguished triangle \cite[Ch.~5, Prop.~3.5.4]{VSF}
$$M(U) \to M(Y) \to M(C)(1)[2] \to M(U)[1].$$
By applying the functor $L\Alb$, we find a morphism of distinguished triangles
\begin{equation} \label{eq:gysin}
\begin{tikzcd}
M(U) \arrow{r}\arrow{d} & M(Y) \arrow{r}\arrow{d} & M(C)(1)[2] \arrow{r}\arrow{d} & M(U)[1] \arrow{d} \\
L\Alb(U) \arrow{r} & L\Alb(Y) \arrow{r} & \Q(1)[2] \arrow{r} & L\Alb(U)[1].
\end{tikzcd}
\end{equation}
Here we used the fact that $L\Alb\big(M(C)(1)\big) \simeq \Q(1)$ \cite[Prop.~8.2.3]{BVK}. Moreover, the morphism
$$M(C)(1) \to L\Alb\big(M(C)(1)\big) \simeq \Q(1)$$
coincides with the projection in \eqref{eq:ckcurve}
$$M(C) \to L_0\Alb(C) \simeq \Q$$
twisted by $\Q(1)$.

Now we apply motivic homology to the distinguished triangles in \eqref{eq:gysin}. Since $h_0(U) \to h_0(Y)$ is surjective \cite[Lem.~4.2]{logmumford} and
$$h_0\big(\Q(1)[2]\big) = \CH_{-1}(\pt) = 0,$$
we obtain a commutative diagram with exact rows
$$\begin{tikzcd}
h_1\big(M(C)(1)[2]\big) \arrow{r}\arrow{d} & h_0(U) \arrow{r}\arrow{d} & h_0(Y) \arrow{r}\arrow{d} & 0 \arrow{d} \\
h_1\big(\Q(1)[2]\big) \arrow{r} & h_0\big(L\Alb(U)\big) \arrow{r} & h_0\big(L\Alb(Y)\big) \arrow{r} & 0.
\end{tikzcd}$$
The first vertical arrow is surjective since it comes from a projection. The middle vertical arrows are given by the Albanese morphisms of $U$ and $Y$. Our assumption $T(U) = 0$ says that the second vertical arrow is injective. Then, by the five lemma, the third vertical arrow is also injective, and hence $T(Y) = 0$.

For (2) $\Rightarrow$ (3), a result of Guletski{\u\i}-Pedrini \cite[Th.~7]{Guletski03} shows that $T(X) = 0$ if and only if $M(X) \in \DMgmab$. By applying several Gysin triangles, we also know that $M(X) \in \DMgmab$ if and only if $M(U) \in \DMgmab$. \qed

\subsection{Proof of (3) \texorpdfstring{$\Rightarrow$}{=>} (1)}
Consider the distinguished triangle
\begin{equation} \label{eq:tria}
M'(U) \to M(U) \to L\Alb(U) \to M'(U)[1].
\end{equation}
Our assumption $p_g(U) = 0$ says that $H_2(U, \Q) = H_2(U, \Q)_{\leq 1}$. Then, by Theorem \ref{thm:real} and Corollary \ref{cor:real}, we have
$$H_i\Big(R^B\big(M'(U)\big)\Big) =
\begin{cases}
H_3(U, \Q) & i = 3 \\
H_4(U, \Q) & i = 4 \\
0 & i < 3 \text{ or } i > 4.
\end{cases}$$

Next, consider the motive $M'(U)^\vee(2)[4]$, whose Betti realization is
$$H_i\Big(R^B\big(M'(U)^\vee(2)[4]\big)\Big) =
\begin{cases}
H_0^{\BM}(U, \Q) & i = 0 \\
H_1^{\BM}(U, \Q) & i = 1 \\
0 & i < 0 \text{ or } i > 1.
\end{cases}$$
It fits in a distinguished triangle
$$L\Alb(U)^\vee(2)[4] \to M^c(U) \to M'(U)^\vee(2)[4] \to L\Alb(U)^\vee(2)[5].$$
Since $L\Alb(U)^\vee(2)[4] \in \DMgmeff$ by Cartier duality \cite[Prop.~4.5.1]{BVK}, we have $M'(U)^\vee(2)[4] \in \DMgmeff$. This allows us to apply the functor $L\Alb$ to $M'(U)^\vee(2)[4]$. By Theorem \ref{thm:real} and Corollary \ref{cor:real}, the morphism
\begin{equation} \label{eq:cons}
M'(U)^\vee(2)[4] \to L\Alb\big(M'(U)^\vee(2)[4]\big)
\end{equation}
induces an isomorphism
$$R^B\big(M'(U)^\vee(2)[4]\big) \xrightarrow{\sim} R^B\Big(L\Alb\big(M'(U)^\vee(2)[4]\big)\Big).$$

We are ready to apply conservativity. Our assumption $M(U) \in \DMgmab$ implies $M^c(U) \in \DMgmab$. Moreover, since $d_{\leq 1}\DMgmeff \subset \DMgmab$, we have $L\Alb(U)^\vee(2)[4] \in  \DMgmab$ and hence $M'(U)^\vee(2)[4] \in \DMgmab$. Then, according to Theorem~\ref{thm:cons}, the morphism \eqref{eq:cons} is itself an isomorphism.

We thus obtain from \eqref{eq:tria} a distinguished triangle
\begin{multline} \label{eq:tria2}
L\Alb\big(M'(U)^\vee(2)[4]\big)^\vee(2)[4] \to M(U) \to L\Alb(U) \\
\to L\Alb\big(M'(U)^\vee(2)[4]\big)^\vee(2)[5].
\end{multline}
Taking motivic homology, we have an exact sequence
$$h_0\Big(L\Alb\big(M'(U)^\vee(2)[4]\big)^\vee(2)[4]\Big) \to h_0(U) \to h_0\big(L\Alb(U)\big),$$
where the second arrow is given by the Albanese morphism of $U$. Hence, to prove $T(U) = 0$, it suffices to show that
$$h_0\Big(L\Alb\big(M'(U)^\vee(2)[4]\big)^\vee(2)[4]\Big) = 0.$$

For this we observe that
$$L\Alb\big(M'(U)^\vee(2)[4]\big) \simeq \bigoplus_{i = 0}^1L_i\Alb\big(M'(U)^\vee(2)[4]\big)[i] \simeq \bigoplus_{i = 0}^1L_i\Alb^c(U)[i].$$
Here we have used the fact that the Hodge realization gives a full embedding $\M_1 \subset \MHS$ \cite[Sect.~10.1.3]{Hodgeiii}. We compute
\begin{align*}
& h_0\Big(L\Alb\big(M'(U)^\vee(2)[4]\big)^\vee(2)[4]\Big) \\
={} & h_0\bigg(\bigoplus_{i = 0}^1\big(L_i\Alb^c(U)[i]\big)^\vee(2)[4]\bigg) \\
={} & \Hom_{\DMgm}\bigg(\Q, \bigoplus_{i = 0}^1\big(L_i\Alb^c(U)[i]\big)^\vee(2)[4]\bigg) \\
={} & \Hom_{\DMgm}\bigg(\bigoplus_{i = 0}^1L_i\Alb^c(U)[i], \Q(2)[4]\bigg) \\
={} & \Hom_{\DMgm}\big(L_0\Alb^c(U), \Q(2)[4]\big) \oplus \Hom_{\DMgm}\big(L_1\Alb^c(U), \Q(2)[3]\big).
\end{align*}
By \cite[Prop.~10.6.2]{BVK}, we have
$$L_0\Alb^c(U) \simeq \begin{cases}
\Q & \text{if $U$ is projective} \\
0 & \text{if not}.
\end{cases}$$
Since
$$\Hom_{\DMgm}\big(\Q, \Q(2)[4]\big) = \CH_{-2}(\pt) = 0,$$
we find in both cases $\Hom_{\DMgm}\big(L_0\Alb^c(U), \Q(2)[4]\big) = 0$.

Further, by \cite[Cor.~12.11.2]{BVK}, the $1$-motive $L_1\Alb^c(U)$ is represented by a two-term complex in degrees $0$ and $-1$
$$\Q^{\oplus r} \to A \otimes \Q,$$
where $A$ is an abelian variety and $r = \#\{\text{connected components of $D$}\} - 1$. In other words, there is an extension of $1$-motives
\begin{equation} \label{eq:ext}
0 \to (A \otimes \Q)[-1] \to L_1\Alb^c(U) \to \Q^{\oplus r} \to 0,
\end{equation}
which yields an exact sequence
\begin{multline*}
\Hom_{\DMgm}\big(\Q, \Q(2)[3]\big)^{\oplus r} \to \Hom_{\DMgm}\big(L_1\Alb^c(U), \Q(2)[3]\big) \\ \to \Hom_{\DMgm}\big((A \otimes \Q)[-1], \Q(2)[3]\big).
\end{multline*}
Since
$$\Hom_{\DMgm}\big(\Q, \Q(2)[3]\big) = \CH_{-2}(\pt, 1) = 0,$$
it suffices to show that $\Hom_{\DMgm}\big((A \otimes \Q)[-1], \Q(2)[3]\big) = 0$.

We may assume $A$ to be the Albanese variety of a smooth projective surface $S$ (which exists by the Lefschetz hyperplane theorem). Recall the Chow-K\"unneth decomposition of $M(S)$ \cite[Th.~3]{Murre90}
$$M(S) \simeq \bigoplus_{i = 0}^4M_i(S)[i].$$
We have $M_{4 - i}(S) \simeq M_i(S)^\vee(2)$ and $M_1(S) \simeq (A \otimes \Q)[-1]$. Hence
\begin{align} \label{eq:cks}
\Hom_{\DMgm}\big((A \otimes \Q)[-1], \Q(2)[3]\big) & = \Hom_{\DMgm}\big(M_1(S), \Q(2)[3]\big) \\
& = \Hom_{\DMgm}\big(\Q, M_3(S)[3]\big) \nonumber \\
& = \CH_0\big(M_3(S)[3]\big) \nonumber \\
& = 0, \nonumber
\end{align}
where the last equality follows again from \cite[Th.~3]{Murre90}. The proof of Theorem \ref{thm:2} is now complete. \qed

\subsection{``Chow-K\"unneth'' decomposition}
Our proof of Theorem~\ref{thm:2} also leads to the following consequence.

\begin{cor}
Assume one of the equivalent conditions in Theorem~\ref{thm:2}. Then $M(U)$ admits a ``Chow-K\"unneth'' decomposition
$$M(U) \simeq \bigoplus_{i = 0}^2 L_i\Alb(U)[i] \oplus \bigoplus_{i = 3}^4 L_{4 - i}\Alb^c(U)^\vee(2)[i].$$
In particular, it is Kimura-O'Sullivan finite.
\end{cor}

\begin{proof}
Assuming $M(U) \in \DMgmab$, we have obtained in \eqref{eq:tria2} a distinguished triangle
\begin{multline*}
\bigoplus_{i = 3}^4 L_{4 - i}\Alb^c(U)^\vee(2)[i] \to M(U) \to \bigoplus_{i = 0}^2 L_i\Alb(U)[i] \\
\to \bigoplus_{i = 3}^4 L_{4 - i}\Alb^c(U)^\vee(2)[i + 1].
\end{multline*}
For the distinguished triangle to split, it suffices to show that
$$\Hom_{\DMgm}\bigg(\bigoplus_{i = 0}^2 L_i\Alb(U)[i], \bigoplus_{i = 3}^4 \big(L_{4 - i}\Alb^c(U)\big)^\vee(2)[i + 1]\bigg) = 0.$$
The left-hand side consists of six direct summands, all of which can be computed explicitly. To keep the paper short we shall only do the most complicated one, that is,
\begin{equation} \label{eq:cku}
\Hom_{\DMgm}\big(L_1\Alb(U)[1], L_1\Alb^c(U)^\vee(2)[4]\big).
\end{equation}

By \cite[Cor.~9.2.3]{BVK}, the $1$-motive $L_1\Alb(U)$ is represented by the two-term complex in degrees $0$ and $-1$
$$0 \to \Alb(U) \otimes \Q.$$
This gives an extension of $1$-motives
\begin{equation} \label{eq:ext2}
0 \to (\G_m \otimes \Q)^{\oplus s}[-1] \to L_1\Alb(U) \to (A' \otimes \Q)[-1] \to 0,
\end{equation}
where $A'$ is the abelian part of the semi-abelian variety $\Alb(U)$. Again we assume $A'$ to be the Albanese variety of a smooth projective surface $S'$, and hence $(A' \otimes \Q)[-1] \simeq M_1(S')$. We also have $(\G_m \otimes \Q)[-1] \simeq \Q(1)$.

Combining \eqref{eq:ext} and \eqref{eq:ext2}, we see that \eqref{eq:cku} sits in the middle of several extensions involving the following four terms:
\begin{enumerate}
\item $\Hom_{\DMgm}\big(M_1(S')[1], M_3(S)[4]\big)$;

\item $\Hom_{\DMgm}\big(M_1(S')[1], \Q(2)[4]\big)$;

\item $\Hom_{\DMgm}\big(\Q(1)[1], M_3(S)[4]\big)$;

\item $\Hom_{\DMgm}\big(\Q(1)[1], \Q(2)[4]\big)$.
\end{enumerate}
The vanishing of the second term is shown in \eqref{eq:cks} (with $S'$ replaced by~$S$). The vanishing of the three other terms follows from the fact that given two Chow motives $M$ and $M'$, we have $\Hom_{\DMgm}\big(M, M'[i]\big) = 0$ for all $i > 0$ \cite[Ch.~5, Cor.~4.2.6]{VSF}. Hence \eqref{eq:cku} vanishes.

Finally, by \cite[Rem.~5.11]{Mazza04}, all elements in $d_{\leq 1}\DMgmeff$ are Kimura-O'Sullivan finite. The last statement follows since Kimura-O'Sullivan finiteness is closed under direct sums and tensor products.
\end{proof}

On the other hand, there exist motives of smooth surfaces which are not Kimura-O'Sullivan finite \cite[Th.~5.18]{Mazza04}.

\bibliography{myref.bib}
\bibliographystyle{alpha}

\end{document}